\documentclass[12pt]{amsart}

\usepackage{amssymb,amsfonts}
\usepackage{amsthm}
\usepackage{graphicx}
\usepackage[all,arc]{xy}
\usepackage[colorlinks=true, allcolors=magenta]{hyperref}
\usepackage{enumerate}
\usepackage{bbm}
\usepackage{dsfont}
\usepackage{mathrsfs}
\usepackage{tikz-cd}
\usetikzlibrary{shapes}
\usepackage{import}
\usepackage{float}
\usepackage{caption}
\usepackage{subcaption}
\usepackage{mathtools}
\restylefloat{table}

\usepackage[left=1in,top=1in,right=1in,bottom=1in]{geometry}

\newtheorem{thm}{Theorem}[section]
\newtheorem{cor}[thm]{Corollary}
\newtheorem{prop}[thm]{Proposition}
\newtheorem{lem}[thm]{Lemma}

\theoremstyle{definition}
\newtheorem{defn}[thm]{Definition}

\newtheorem{exmp}[thm]{Example}

\newtheorem{rem}[thm]{Remark}
\newtheorem*{thm1.2}{\textrm{Theorem 1.2}}

\theoremstyle{remark}

\newtheorem{sch}[thm]{Scholium}

\newcommand{\Z}{\mathbb{Z}}
\newcommand{\Q}{\mathbb{Q}}

\newcommand{\C}{\mathbf{C}}

\newcommand{\R}{\mathbf{R}}

\newcommand{\Aut}{\operatorname{Aut}}

\renewcommand{\emptyset}{\varnothing}

\newcommand{\G}{\mathbf{G}}
\newcommand{\Hom}{\operatorname{Hom}}

\newcommand{\gr}{\textup{Gr}}

\newcommand{\val}{\operatorname{val}}
\newcommand{\Var}{\textup{Var}}

\newcommand{\B}{\mathbf{B}}

\newcommand{\bfH}{\mathbf{H}}

\usepackage[colorinlistoftodos]{todonotes}
\newcommand{\simp}{\operatorname{Simp}}

\def\C{\mathbb{C}}

\def\G{\mathbf{G}}

\def\Q{\mathbb{Q}}
\def\R{\mathbb{R}}

\def\Z{\mathbb{Z}}

\def\calM{\mathcal{M}}

\def\calP{\mathcal{P}}

\def\calX{\mathcal{X}}

\newcommand*\circled[1]{\tikz[baseline=(char.base)]{\node[shape=circle,draw,inner sep=2pt, scale=0.7] (char) {#1};}}
\newcommand{{\circone}}{\circled{$1$}}
\newcommand{{\circtwo}}{\circled{$2$}}
\newcommand{{\circthree}}{\circled{$3$}}
\newcommand{{\circn}}{\circled{$n$}}

\newtheorem*{thm*}{Theorem}

\makeatletter
\let\c@equation\c@thm
\makeatother
\numberwithin{equation}{section}

\graphicspath{ {Hassett-figs/} }
\title[Topology of tropical moduli spaces of weighted stable curves]{Topology of tropical moduli spaces of weighted stable curves in higher genus}

\author[S. Kannan]{Siddarth Kannan}\address{Department of Mathematics, Brown University, Providence, RI 02906}
\email{\url{siddarth_kannan@brown.edu}}
\author[S. Li]{Shiyue Li}\address{Department of Mathematics, Brown University, Providence, RI 02906}
\email{\url{shiyue_li@brown.edu}}
\author[S. Serpente]{Stefano Serpente}\address{Dipartimento di Matematica e Fisica, Università Roma Tre, Rome, I-00146}\email{\url{stefano.serpente@uniroma3.it}}
\author[C. Yun]{Claudia He Yun}\address{Department of Mathematics, Brown University, Providence, RI 02906}
\email{\url{he_yun@brown.edu}}

\begin{document}

\begin{abstract}
Given integers $g \geq 0$, $n \geq 1$, and a vector $w \in (\Q \cap (0, 1])^n$ such that ${2g - 2 + \sum w_i > 0}$, we study the topology of the moduli space $\Delta_{g, w}$ of $w$-stable tropical curves of genus $g$ with volume 1. The space $\Delta_{g, w}$ is the dual complex of the divisor of singular curves in Hassett's moduli space of $w$-stable genus $g$ curves $\overline{\mathcal{M}}_{g, w}$. When $g \geq 1$, we show that $\Delta_{g, w}$ is simply connected for all values of $w$. We also give a formula for the Euler characteristic of $\Delta_{g, w}$ in terms of the combinatorics of $w$.
\end{abstract}
	
\maketitle
\tableofcontents

\section{Introduction}
Fix integers $g \geq 0$, $n \geq 1$ and a vector of rational weights $w \in (\Q \cap (0, 1])^n$ satisfying
\begin{equation}
\label{eq:weight-genus-condition}
2g - 2 + \sum_{i = 1}^{n} w_i > 0. 
\end{equation}
We study the topology of the moduli space $\Delta_{g, w}$ of $w$-stable tropical curves of genus $g$ and volume $1$. Here a $w$-stable tropical curve is a pair $(\G, \ell)$ where $\G$ is a $w$-stable graph (see (\ref{wStability})) and $\ell: E(\G) \to \R_{\geq 0}$ is a length function; the \textit{volume} is the sum of edge lengths in $\mathbf{G}$, i.e., $\sum_{e \in E(\G)} \ell(e)$. Following the work of Ulirsch ~\cite{ulirsch14tropical} and Chan, Galatius, and Payne ~\cite{CGP1}, the space $\Delta_{g, w}$ can be realized as the dual complex of the normal crossings divisor ${\overline{\mathcal{M}}_{g, w} \smallsetminus \mathcal{M}_{g, w}}$ on Hassett's moduli stack $\overline{\mathcal{M}}_{g, w}$; here $\mathcal{M}_{g, w}$ denotes the dense open substack of $\overline{\mathcal{M}}_{g, w}$ parameterizing smooth but not necessarily distinctly marked curves. Our first main theorem is that $\Delta_{g, w}$ is simply connected for $g \geq 1$.
\begin{thm}
\label{thm:sc}
For any $g,n \geq 1$ and $w \in (\Q \cap (0, 1])^n$, the space $\Delta_{g, w}$ is simply connected.
\end{thm}
Our second result is a calculation of the Euler characteristic of $\Delta_{g, w}$ in terms of the \textit{top weight} Euler characteristics of the moduli spaces $\calM_{g, r}$ of smooth $r$-marked algebraic curves of genus $g$; see essential background on weight filtrations and mixed Hodge structures in Section \ref{Motivation}. Set
\[[n] := \{1, \ldots, n\}. \]
We call a partition $P_1 \sqcup \cdots \sqcup P_r \vdash [n]$ \textbf{\textit{$w$-admissible}} if $\sum_{i \in P_j} w_i \leq 1$
for all $1 \leq j \leq r$. Let $N_{r, w}$ denote the number of $w$-admissible partitions of $[n]$ with $r$ parts.
\begin{thm}
\label{thm:main-euler}
Let $W = W_1 \subset \cdots \subset W_{6g-6+2r} \subseteq H^{\ast}(\calM_{g, r}, \Q)$ be the weight filtration of the rational singular cohomology of the moduli stack $\calM_{g, r}$ and denote by $\chi^{W}_{6g-6+2r}$ the Euler characteristic of the top graded piece
\[\gr^W_{6g - g + 2r}H^*(\mathcal{M}_{g, r}; \Q) = W_{6g - 6 + 2r}/W_{6g - 7 + 2r}  \]
of the weight filtration. Then
\[\chi(\Delta_{g, w}) = 1  -  \sum_{r = 1}^{n} N_{r, w} \cdot \chi^{W}_{6g - 6 + 2r} (\mathcal{M}_{g, r}).  \]
\end{thm}
    
    A generating function for the numbers $\chi^{W}_{6g - 6 + 2r} (\mathcal{M}_{g, r})$ is given in \cite{CFGP}; together with their result, Theorem \ref{thm:main-euler} allows for the computer-aided calculation of $\chi(\Delta_{g, w})$ for arbitrary $g$ and $w$. In \cite[Corollary 8.1]{CFGP}, the authors give a closed form in the case when the number of marked points is large. For $r > g + 1$, 
        \[
        \chi^{W}_{6g - 6 + 2r}(\calM_{g, r}) = (-1)^{r+1} \frac{(g + r - 2)!}{g!} B_g,
        \] where $B_g$ is the $g$-th Bernoulli number, characterized by 
        \[
        \frac{t}{e^{t}-1} = \sum_{\ell=0}^{\infty} B_\ell \frac{t^{\ell}}{\ell!}.
        \] Substituting into Theorem \ref{thm:main-euler} yields the following closed form. 
    \begin{cor}
    \label{cor:euler-bernoulli}
        Given $g \ge 0$ and a weight vector $w$ satisfying Equation \ref{eq:weight-genus-condition} and that $N_{r, w} = 0$ for $r \le g+ 1$,
        the Euler characteristic of $\Delta_{g, w}$ is
        \begin{align*}
         \chi(\Delta_{g, w}) &= 1 + \sum_{r = 1}^{n} N_{r, w} (-1)^{r} \frac{(g + r - 2)!}{g!} B_g.
        \end{align*}
        
    \end{cor}
    
    Let $S(m, r)$ denote the number of $r$-partitions of $[m]$ for $m \ge 1$ and $r \ge 0$; these are called \textit{the Stirling numbers of the second kind}. Expanding the Bernoulli number $B_g$ (see  \cite{apostol1998introduction}) in terms of Stirling numbers
        \[
        B_g = \sum_{\ell = 0}^{g} (-1)^{\ell} \frac{\ell!}{\ell + 1} S(g, \ell),
        \] 
    we obtain the following closed form for the Euler characteristic of $\Delta_{g, w}$ for \textit{heavy/light} weights.  
    \begin{cor}
    Given a heavy/light weight vector $w = (1^{(n)}, \varepsilon^{(m)})$ where $n \ge g+1$, $m > 0$, and $0 < \varepsilon < 1/m$,
    \[
     \chi(\Delta_{g, w}) = 1 + \sum_{r = 1}^{m} \sum_{\ell = 0}^{g} (-1)^{n + r+\ell} \frac{(g+n + r-2)! \ell!}{g!(\ell + 1)} S(m, r) S(g, \ell).
    \]
    \end{cor}
  
Using this corollary above, we compute explicitly some of the Euler characteristics of $\Delta_{g, (1^{(n)}, \varepsilon^{(m)})}$ in Table \ref{table:heavy-light}. 
\begin{table}[h]
\begin{tabular}{ccccccccccc}
\cline{1-5} \cline{7-11}
\multicolumn{1}{|c|}{\begin{tabular}[c]{@{}c@{}}$g=0$\end{tabular}} & \multicolumn{1}{c|}{$m=1$}    & \multicolumn{1}{c|}{$m=2$}    & \multicolumn{1}{c|}{$m=3$}      & \multicolumn{1}{c|}{$m=4$}      & \multicolumn{1}{c|}{} & \multicolumn{1}{c|}{\begin{tabular}[c]{@{}c@{}}$g=1$\end{tabular}} & \multicolumn{1}{c|}{$m=1$}   & \multicolumn{1}{c|}{$m=2$}   & \multicolumn{1}{c|}{$m=3$}     & \multicolumn{1}{c|}{$m=4$}    \\ \cline{1-5} \cline{7-11} 
\multicolumn{1}{|c|}{$n=2$}                                                                                & \multicolumn{1}{c|}{-}    & \multicolumn{1}{c|}{$2$}    & \multicolumn{1}{c|}{$0$}      & \multicolumn{1}{c|}{$2$}      & \multicolumn{1}{l|}{} & \multicolumn{1}{c|}{$n=2$}                                                                                & \multicolumn{1}{c|}{$2$}   & \multicolumn{1}{c|}{$-1$}  & \multicolumn{1}{c|}{$5$}     & \multicolumn{1}{c|}{$-7$}   \\ \cline{1-5} \cline{7-11} 
\multicolumn{1}{|c|}{$n=3$}                                                                                & \multicolumn{1}{c|}{$3$}    & \multicolumn{1}{c|}{$-3$}   & \multicolumn{1}{c|}{$9$}      & \multicolumn{1}{c|}{$-15$}    & \multicolumn{1}{c|}{} & \multicolumn{1}{c|}{$n=3$}                                                                                & \multicolumn{1}{c|}{$-2$}  & \multicolumn{1}{c|}{$10$}  & \multicolumn{1}{c|}{$-26$}   & \multicolumn{1}{c|}{$82$}   \\ \cline{1-5} \cline{7-11} 
\multicolumn{1}{|c|}{$n=4$}                                                                                & \multicolumn{1}{c|}{$-5$}   & \multicolumn{1}{c|}{$19$}   & \multicolumn{1}{c|}{$-53$}    & \multicolumn{1}{c|}{$163$}    & \multicolumn{1}{c|}{} & \multicolumn{1}{c|}{$n=4$}                                                                                & \multicolumn{1}{c|}{$13$}  & \multicolumn{1}{c|}{$-47$} & \multicolumn{1}{c|}{$193$}   & \multicolumn{1}{c|}{$-767$} \\ \cline{1-5} \cline{7-11} 
\multicolumn{1}{|c|}{$n=5$}                                                                                & \multicolumn{1}{c|}{$25$}   & \multicolumn{1}{c|}{$-95$}  & \multicolumn{1}{c|}{$385$}    & \multicolumn{1}{c|}{$-1535$}  & \multicolumn{1}{c|}{} & \multicolumn{1}{c|}{$n=5$}                                                                                & \multicolumn{1}{c|}{$-59$} & \multicolumn{1}{c|}{$301$} & \multicolumn{1}{c|}{$-1499$} & \multicolumn{1}{c|}{$7501$} \\ \cline{1-5} \cline{7-11} 
                                                                                                         &                             &                             &                               &                               &                       &                                                                                                         &                            &                            &                              &                             \\ \cline{1-5} \cline{7-11} 
\multicolumn{1}{|c|}{\begin{tabular}[c]{@{}c@{}}$g=2$\end{tabular}} & \multicolumn{1}{c|}{$m=1$}    & \multicolumn{1}{c|}{$m=2$}    & \multicolumn{1}{c|}{$m=3$}      & \multicolumn{1}{c|}{$m=4$}      & \multicolumn{1}{c|}{} & \multicolumn{1}{c|}{\begin{tabular}[c]{@{}c@{}}$g=3$\end{tabular}} & \multicolumn{1}{c|}{$m=1$}   & \multicolumn{1}{c|}{$m=2$}   & \multicolumn{1}{c|}{$m=3$}     & \multicolumn{1}{c|}{$m=4$}    \\ \cline{1-5} \cline{7-11} 
\multicolumn{1}{|c|}{$n=3$}                                                                                & \multicolumn{1}{c|}{$3$}    & \multicolumn{1}{c|}{$-7$}   & \multicolumn{1}{c|}{$33$}     & \multicolumn{1}{c|}{$-127$}   & \multicolumn{1}{c|}{} & \multicolumn{1}{c|}{$n=4$}                                                                                & \multicolumn{1}{c|}{$1$}   & \multicolumn{1}{c|}{$1$}   & \multicolumn{1}{c|}{$1$}     & \multicolumn{1}{c|}{$1$}    \\ \cline{1-5} \cline{7-11} 
\multicolumn{1}{|c|}{$n=4$}                                                                                & \multicolumn{1}{c|}{$-9$}   & \multicolumn{1}{c|}{$51$}   & \multicolumn{1}{c|}{$-249$}   & \multicolumn{1}{c|}{$1251$}   & \multicolumn{1}{c|}{} & \multicolumn{1}{c|}{$n=5$}                                                                                & \multicolumn{1}{c|}{$1$}   & \multicolumn{1}{c|}{$1$}   & \multicolumn{1}{c|}{$1$}     & \multicolumn{1}{c|}{$1$}    \\ \cline{1-5} \cline{7-11} 
\multicolumn{1}{|c|}{$n=5$}                                                                                & \multicolumn{1}{c|}{$61$}   & \multicolumn{1}{c|}{$-359$} & \multicolumn{1}{c|}{$2161$}   & \multicolumn{1}{c|}{$-12959$} & \multicolumn{1}{c|}{} & \multicolumn{1}{c|}{$n=6$}                                                                                & \multicolumn{1}{c|}{$1$}   & \multicolumn{1}{c|}{$1$}   & \multicolumn{1}{c|}{$1$}     & \multicolumn{1}{c|}{$1$}    \\ \cline{1-5} \cline{7-11} 
\multicolumn{1}{|c|}{$n=6$}                                                                                & \multicolumn{1}{c|}{$-419$} & \multicolumn{1}{c|}{$2941$} & \multicolumn{1}{c|}{$-20579$} & \multicolumn{1}{c|}{$144061$} & \multicolumn{1}{c|}{} & \multicolumn{1}{c|}{$n=7$}                                                                                & \multicolumn{1}{c|}{$1$}   & \multicolumn{1}{c|}{$1$}   & \multicolumn{1}{c|}{$1$}     & \multicolumn{1}{c|}{$1$}    \\ \cline{1-5} \cline{7-11} 
\end{tabular}
\caption{Euler characteristics of $\Delta_{g,(1^{(n)},\varepsilon^{(m)})}$ for $g=0,1,2,3$ and some $(n,m)$ where $n\geq g+1$ and $m>0$. Note that when $g=0$, we start with $n=2$ since the space $\Delta_{0,(1,\varepsilon^{(m)})}$ is empty; when $g = 0, n=2,m=1$, $\Delta_{0, (1, 1, \varepsilon)}$ is also empty.}
\label{table:heavy-light}
\end{table}

\subsection{Motivation}\label{Motivation}
    Throughout the paper, we work over the complex numbers $\C$. The Deligne-Mumford-Knudsen compactification $\overline{\calM}_{g, n}$ is a toroidal compactification of the moduli stack $\calM_{g,n}$; the toroidal structure comes from the fact that the boundary divisor ${\overline{\calM}_{g, n} \smallsetminus \calM_{g, n}}$ has normal crossings. As a Deligne-Mumford stack, the rational cohomology of $\calM_{g, n}$ carries a mixed Hodge structure; see \cite{deligne1971theorie3}. That is, there is a weight filtration
    \[ W_1 \subset \cdots \subset W_{6g - g + 2n} = H^{*}(\mathcal{M}_{g, n}; \Q) \]
    such that, for each $j$, the quotient
    \[\gr^W_{i} H^j(\mathcal{M}_{g, n}, \Q)  = W_i\cap H^j(\mathcal{M}_{g, n}; \Q) /W_{i - 1} \cap H^j(\mathcal{M}_{g, n}; \Q)   \]
    carries a pure Hodge structure of weight $i$.
    The top graded piece of the weight filtration of $\calM_{g, n}$ can be identified with the reduced homology of the dual complex of the divisor $\overline{\calM}_{g, n} \smallsetminus \calM_{g, n}$, up to a degree shift. As discussed in ~\cite{CGP1, CGP2}, the dual complex of this divisor may be identified with the tropical moduli space $\Delta_{g, n}$, furnishing isomorphisms
    \[ \widetilde{H}_{j - 1}(\Delta_{g, n}; \Q) \cong \gr^W_{6g - 6 + 2n} H^{6g - 6 + 2n - j}(\mathcal{M}_{g, n}; \Q).  \]
    
    In \cite{HASSETT2003316}, Hassett introduced the moduli stack $\overline{\calM}_{g, w}$ as an alternative compactification of $\calM_{g, n}$: in $\overline{\calM}_{g, w}$, marked points are allowed to coincide if the sum of the corresponding entries of $w$ is no greater than $1$. Thus $\overline{\calM}_{g, w}$ contains an open substack $\calM_{g, w}$ parameterizing smooth, but not necessarily distinctly marked algebraic curves of genus $g$, and we have the containments ${\calM_{g, n} \subset \calM_{g, w} \subset \overline{\calM}_{g, w}}$. Although the embedding $\calM_{g, n} \subset \overline{\calM}_{g, w}$ is no longer toroidal, $\overline{\calM}_{g, w} \smallsetminus \calM_{g, w}$ is still a normal crossings divisor, and the dual complex of this divisor has a natural modular interpretation as the moduli space $\Delta_{g, w}$ of tropical $w$-stable curves of volume $1$ and genus $g$; see \cite{ulirsch14tropical, CHMR2014moduli}.
    Therefore, one has isomorphisms
    \[ \widetilde{H}_{j - 1}(\Delta_{g, w}; \Q) \cong \gr^W_{6g - 6 + 2n} H^{6g - 6 + 2n - j}(\mathcal{M}_{g, w}; \Q),  \]
    identifying the reduced rational homology of $\Delta_{g, w}$ with the top graded piece of the rational cohomology of $\mathcal{M}_{g, w}$.
    
\subsection{Previous work}
    This work benefits from and builds on previous work of many authors on the topology of tropical moduli spaces, which we summarize here.
    
    When $g = 0$, 
    the complex $\Delta_{0, w}$ may be identified with various objects whose homotopy types are known. 
    \begin{enumerate}
        \item When $w = (1^{(n)})$, Vogtmann showed that $\Delta_{0, n}$ is homotopic to a wedge of $(n - 2)!$ spheres of dimension $n-4$, by identifying it as the link of a vertex of a quotient simplicial complex by the outer automorphisms of a finitely generated free group; see \cite{Culler1986ModuliOG, vogtmann_1990}. In \cite{Robinson1996TheTR}, Robinson and Whitehouse gave a different proof of the same result, by contracting a large subcomplex $X_{0, n}$ of $\Delta_{0, n}$.
        
        \item When $w$ is \textit{heavy/light}, i.e. $w = (1^{(n)}, \varepsilon^{(m)})$ for $\varepsilon < 1/m$, Cavalieri, Hampe, Markwig, and Ranganathan in \cite{CHMR2014moduli} identified $\Delta_{0, w}$ with the link at the origin of the Bergman fan of a graphic matroid, thereby deriving that $\Delta_{0, w}$ is homotopic to a wedge of $(n-2)!(n-1)^{m}$ spheres of dimension $n + m - 4$, using \cite{ARDILA200638}. In \cite{CMPRS}, Cerbu, Marcus, Peilen, Ranganathan, and Salmon rederived this result using Vogtmann's result on $\Delta_{0, n}$ and contracting a large subcomplex. 
        
        \item When $w$ has at least two weight-$1$ entries, Cerbu et al. in \cite{CMPRS} showed that $\Delta_{0, w}$ is homotopic to a wedge of spheres of possibly varying dimensions, by identifying a large contractible subcomplex and using known results on homotopy types of subspace arrangements. The authors also provided infinite families of $w$ where $\Delta_{0, w}$ is disconnected, and examples where $\pi_{1}(\Delta_{0, w}) = \Z/2\Z$.  
        In the latter scenario, the authors proved that the universal cover has the homotopy type of a wedge of spheres. 
    \end{enumerate}
    
    For higher values of $g$, the following results are known. 
    \begin{enumerate}
        \item When $w = (1^{(n)})$, Chan, Galatius, and Payne showed in \cite{CGP2} that $\Delta_{1, n}$ is homotopic to $\frac{1}{2}(n-1)!$ spheres of dimension $n - 1$. In \cite{chan2015topology}, Chan independently showed that the reduced integral homology $\widetilde{H}_{\ast}(\Delta_{2, n}; \Z)$ is supported in the top two degrees and that a subcomplex of $\Delta_{2,n}$ has torsion in high degrees. Chan also computed the reduced rational homology $\widetilde{H}_{\ast}(\Delta_{2, n}; \Q)$ for $n \le 8$.
        For higher genera, Chan, Galatius, and Payne showed that $\Delta_{g, n}$ is at least $(n-3)$-connected \cite{CGP2}.
        \item When $w$ has at least two weight-$1$ entries, \cite{CMPRS} leveraged a relation between $\Delta_{0, w}$ and $\Delta_{1, w}$ to prove that $\Delta_{1, w}$ is homotopic to a wedge of spheres. 
        \item When $w = (1^{(n)}, \varepsilon^{(m)})$ is heavy/light, the same authors showed that $\Delta_{1, w}$ is homotopic to ${\frac{1}{2}(n-1)!n^m}$ spheres of dimension $n + m - 1$. 
    \end{enumerate}
    Most recently, in \cite{allcock}, Allcock, Corey, and Payne showed that $\Delta_{g}$ and $\Delta_{g, n}$ are simply connected for ${(g, n) \ne (0, 4), (0, 5)}$. They give two proofs of this result; one relies on a celluar approximation theorem in dimension 1 for symmetric CW-complexes, and the other, suggested by A. Putman, uses Harer's result ~\cite{Harer} that Harvey's complex of curves $\mathcal{C}_{g, n}$ is simply connected, together with the fact that $\Delta_{g, n}$ is homeomorphic to the quotient of the complex $\mathcal{C}_{g, n}$ by the action of the pure mapping class group. In this paper, we use a similar technique as in the first proof of \cite{allcock} and deduce simple connectedness for $\Delta_{g,w}$. 
    The working framework in this paper is based on graph categories and symmetric $\Delta$-complexes, heavily used in \cite{CGP1, CGP2}. The main tools are the contractibility criterion developed in \cite{CGP2} and the Grothendieck group of varieties. 
\begin{rem}
    The tropical Hassett space $\Delta_{g,w}$ is related to several complexes studied in the context of geometric group theory, notably Harvey's complex of curves $\mathcal{C}_{g,n}$ ~\cite{harvey1981boundary} and Hatcher's complex of sphere systems $\mathcal{S}_{g,n}$ ~\cite{hatcher1995}. In either case, one can define a subcomplex parameterizing $w$-stable collections, i.e., those collections of curves, respectively $2$-spheres, whose dual graph is $w$-stable in the sense of (\ref{wStability}) below. Then $\Delta_{g,w}$ may be realized as the quotient of either subcomplex by a suitable group action. In the case of $\mathcal{C}_{g,n}$, the action is by the pure mapping class group, and in the case of $\mathcal{S}_{g,n}$, the action is by the quotient of the pure mapping class group by the normal subgroup generated by Dehn twists. A proof that either of these subcomplexes is simply connected would lead to an alternate proof of simply connectedness of $\Delta_{g,w}$, via the approach of Putman mentioned above. We also refer any interested reader to the first version of this preprint on arXiv, where we employ the Seifert-van Kampen theorem for CW complexes inductively to prove the simple connectivity of $\Delta_{g, w}$; see \cite{kannan2020topology}. 
\end{rem}    
    
\subsection*{Acknowledgements} We are grateful to Melody Chan for suggesting a motivic approach to the proof of Theorem \ref{thm:main-euler}, to Sam Payne for help with questions about mixed Hodge structures and connections to geometric group theory, and to Sam Freedman for many useful conversations related to this work. We also thank the organizers of the 2020 Summer Tropical Algebraic Geometry Online SeminAUR (STAGOSAUR)/Algebraic and Tropical Online Meetings (ATOM) for great learning opportunities and for bringing us together. SK was supported by an NSF Graduate Research Fellowship.
      
\section{Background}

\subsection{The graph categories $\Gamma_{g, w}$.}
\label{subsec:graph-cat}
Given integers $g \geq 0, n \geq 1$ and a vector $w \in (\Q \cap (0, 1])^n$ of rational numbers satisfying
\begin{equation}\label{gA}
    2g - 2 + \sum_{i = 1}^{n} w_i > 0,
\end{equation}
we define a graph category $\Gamma_{g, w}$. All graphs considered in this paper allow loops and parallel edges. First, a weighted $n$-marked graph $\G$ is a triple ${\G = (G, m, h)}$ consisting of a finite connected graph $G$ together with a marking function $m : [n] \to V(G)$ and a vertex weight function $h: V(G) \to \Z_{\geq 0}$. We say $\G$ is $w$-stable if it satisfies the stability condition 
\begin{equation}\label{wStability}
2 h(v) - 2 + \val(v) + \sum_{i \in m^{-1}(v)} w_i > 0,
\end{equation}
for all $v \in V(G)$, where $\val(v)$ denotes the valence of $v$ in $G$, i.e., the number of half edges incident to $v$; thus a loop contributes twice to the valence of its vertex. The genus of $\G$ is defined as \[\mathbf{g}(\G) :=b^1(G) + \sum_{v \in V(G)} h(v), \] where $b^1(G) = |E(G)| - |V(G)| + 1$ is the first Betti number of $G$. The objects in $\Gamma_{g,w}$ are $w$-stable weighted $n$-marked graphs of genus $g$.

The morphisms in $\Gamma_{g, w}$ are maps that factor as compositions of isomorphisms and edge contractions. To be precise, an \textit{isomorphism} $\varphi: \G \to \G'$ where $\G = (G, m, h)$ and ${\G' = (G', m', h')}$ is an isomorphism $\varphi: G \to G'$ such that $m' = \varphi \circ m$ and $h' \circ \varphi = h$. An \textit{edge contraction} $c: \G \to \G/e$ of an edge $e$ in $G$ is given by removing $e$ and identifying its two endpoints if $e$ is not a loop, and by removing $e$ and increasing the weight of its base vertex by one if $e$ is a loop. We say two graphs have the same \textit{combinatorial type} if they are isomorphic in $\Gamma_{g, w}$.

We will say that $\G'$ is an \textit{uncontraction} of a graph $\G$ if $\G'$ contracts to $\G$ after a series of edge contractions. One can alternatively describe $\Gamma_{g, w}$ as the full subcategory of the category $\Gamma_{g, n}$ defined in ~\cite{CGP2} whose objects are graphs in $\Gamma_{g, n}$ that satisfy (\ref{wStability}); when $w = (1^{(n)})$, $\Gamma_{g, n} = \Gamma_{g, w}$. To work with a small category, we hereafter tacitly replace $\Gamma_{g, n}$ with a choice of skeleton thereof. This also induces a choice of skeleton of $\Gamma_{g, w}$ for all weight vectors $w$. Abusing notation, we will write $\G \in \Gamma_{g, w}$ for $\G \in \mathrm{Ob}(\Gamma_{g, w})$. We now give the definition of the tropical moduli space $\Delta_{g, w}$ using $\Gamma_{g, w}$.

\subsection{Description of $\Delta_{g, w}$ as a symmetric $\Delta$-complex}
\label{sec:sym} The space $\Delta_{g, w}$ is the geometric realization of a symmetric $\Delta$-complex, in the sense of ~\cite{CGP1}. Let $I$ be the category having one object for each finite set
\[[p] := \begin{cases}\{0, \ldots, p\} & \text{for }p \ge 0, \\
\varnothing & p = -1.
\end{cases}
\]
and morphisms consisting of all injections.
A symmetric $\Delta$-complex $X$ is a functor ${X: I^\mathrm{op} \to \mathsf{Sets}}$, and a morphism of symmetric $\Delta$-complexes is a natural transformation of functors. 
For simplicity, we write $X_p$ for $X([p])$. There is a geometric realization functor associating a topological space to each symmetric $\Delta$-complex $X$, which we describe below. Each injection $\iota:[p] \to [q]$ induces a map on standard simplices $\iota_*: \sigma^p \to \sigma^q$, defined by \[\iota_*\left(\sum_{i=0}^p t_ie_i\right) = \sum_{i=0}^q\left(\sum_{j\in\iota^{-1}(i)}t_j\right)e_i.\] The geometric realization of $X$ is then defined as
\[|X| := \left(\coprod_{p \geq 0} X_p \times \sigma^p\right) \bigg/ \sim \]
where the equivalence relation $\sim$ is generated by relations of the form
$(X(\iota)(x), a) \sim (x, \iota_*(a)),$
whenever $\iota \in \Hom_I\left([p], [q]\right)$, $x \in X_q$, and $a \in \sigma^p$. The $k$-skeleton of $|X|$ is \[|X|^{(k)} = \left(\coprod_{p = 0}^k X_p \times \sigma^p\right) \bigg/ \sim\] with the same equivalence relation.

We will use $\Delta_{g, w}$ both for the functor and its geometric realization when there is no confusion. We set
\[\Delta_{g, w}([p]) = \{(\G, \tau) \mid \G \in \Gamma_{g, w},|E(\G)| = p + 1,\, \tau: E(\G) \to [p] \text{ a bijection}  \}/\sim,   \]
where $(\G, \tau) \sim (\G', \tau')$ if and only if there exists a $\Gamma_{g, w}$-isomorphism $\varphi: \G \to \G'$ such that the diagram
\[ \begin{tikzcd}
&E(\G) \arrow[rr, "\varphi"] \arrow[dr, "\tau"] & &E(\G') \arrow[dl, "\tau'"']\\
& &\lbrack p \rbrack  &
\end{tikzcd}\]
commutes. We put $[\G, \tau]$ for the equivalence class of $(\G, \tau)$. On morphisms, we define $\Delta_{g, w}$ as follows: given an injection $\iota: [p] \to [q]$ and $[\G, \tau] \in \Delta_{g, w}([q])$, we set $\bfH$ to be the graph obtained from $\G$ by contracting all edges which are not labelled by $\tau^{-1}(\iota([p]))$, and $\pi: E(\bfH) \to [p]$ to be the unique edge-labelling of $\bfH$ which preserves the order of the remaining edges. It is known that the order of contraction of edges does not affect the final result. Then $\Delta_{g, w}(\iota)([\G, \tau]) = [\bfH, \pi]$.
\begin{rem}
    There is an alternate definition of $\Delta_{g,w}$ as a colimit of maps of simplices; see \cite[Section 1]{CMPRS}. We briefly summarize the intuition here.
    
    As defined in the introduction, a $w$-stable tropical curve of volume 1 is a pair $(\G, \ell)$ where $\G \in \Gamma_{g,w}$ and $\ell: E(\G) \to \R_{\geq 0}$ is a length function satisfying $\sum_{e \in E(\G)} \ell(e) = 1$. If $\ell(e) = 0$ for some $e$, then $(\G, \ell)$ is identified with $(\G/e, \ell|_{\G/e})$, where $\ell|_{\G/e}$ is the restriction of $\ell$ to $E(\G/e) = E(\G) \smallsetminus\{e\}$. A point $x \in \Delta_{g,w}$ can be identified with a $w$-stable tropical curve of volume $1$. Suppose $x$ is contained in a $p$-simplex corresponding to an edge-labelled graph $[\G, \tau]$, with coordinates $(a_0, \ldots, a_p)$. Then $x$ corresponds to the tropical curve $(\G, \ell)$ where $\ell(\tau^{-1}(i)) = a_i$ for each $i \in [p]$. In general, the point $x$ may be contained in more than one simplex, but there will be a unique simplex which contains $x$ in its interior. This corresponds to the unique representative of the resulting tropical curve which has a strictly positive length function.
\end{rem}

\begin{exmp} Let $g=1$ and $w=(\varepsilon^{(3)})$ for $0 < \varepsilon < 1/3$. The combinatorial types in $\Gamma_{1,w}$ are edge contractions of the type shown in the interior of the left triangle in Figure \ref{fig:delta2-1eee}. The 0-skeleton of $\Delta_{1, w}$ is a single point; the 1-skeleton is three half-edges glued at the point. The whole space is a hollow tetrahedron, which is homeomorphic to $S^2$. 

\begin{figure}[h]
    \includegraphics[scale=0.6]{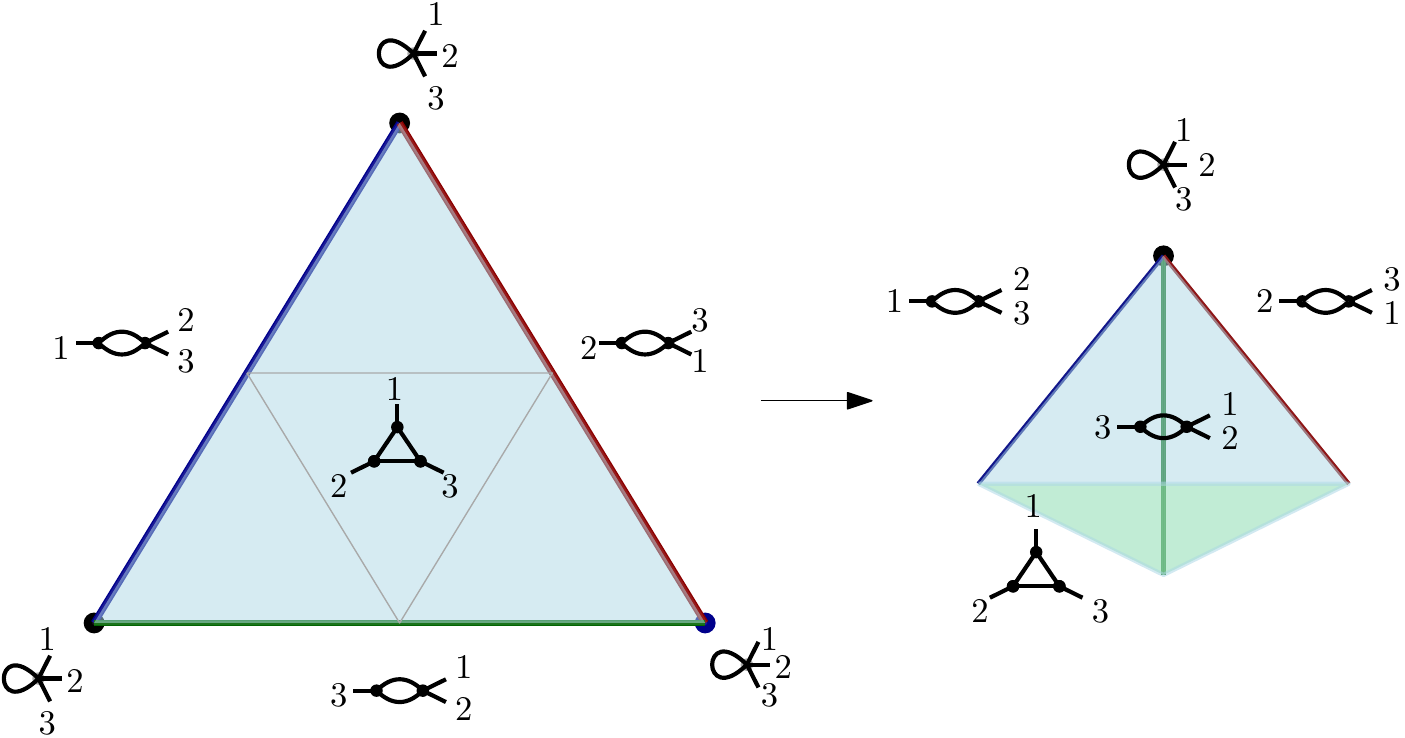}
    \caption{
    The geometric realization of $\Delta_{1, w}$ is the (hollow) tetrahedron on the right, obtained by folding the $2$-simplex on the left along the interior grey lines.
    }
    \label{fig:delta2-1eee}
\end{figure}
\end{exmp}

\section{Proof that $\Delta_{g,w}$ is simply connected}
\label{sec: delta_gw is simply connected}
Let $g \geq 1$, $n > 0$, and fix a weight vector $w \in (\Q \cap (0, 1])^n$. We will prove that $\pi_1(\Delta_{g, w})$ is trivial. Our strategy is to show that the 1-skeleton of $\Delta_{g,w}$ is contained in a contractible subcomplex. We first recall briefly the contractibility criterion developed in Section 4 of \cite{CGP2}.

Let $X$ be a symmetric $\Delta$-complex (representing both the functor and the geometric realization thereof when no confusion arises), and continue to denote by $X_i$ the set $X([i])$.
\begin{defn}
A property on $X$ is a subset of the vertices $P \subseteq X_0$.
\end{defn}

Writing $\simp(X) = \coprod_{p \geq 0}X_p$ for the set of all simplices of $X$, and given a subset of vertices $P \subseteq X_0$, we define 
\[P(X) := \{\sigma \in \simp(X): \sigma \text{ has a vertex in }P\}.\] Elements of $P(X)$ are called $P$-simplices. Given an integer $i\geq 0$, let $X_{P,i}$ denote the subcomplex of $X$ generated by the set of $P$-simplices with at most $i$ non-$P$ vertices. Write $X_P = X_{P,\infty}$ to be the subcomplex generated by all $P$-simplices. 

Let $\B \in \Gamma_{g,w}$ be the graph consisting of one edge connecting two vertices $v_1$ and $v_2$, such that $h(v_1) = 1$ and $m^{-1}(v_1) = \emptyset$. Note that when $g=1$ with $\sum_{i=1}^n w_i > 1$, or when $g \ge 2$, $\B$ is stable. Given a graph $\G \in \Gamma_{g,w}$, we call an edge $e\in E(\G)$ a \textit{1-end} if contracting all edges of $\G$ except for $e$ yields $\B$.

\begin{lem}
Let $g=1$ with $\sum_{i=1}^n w_i > 1$, or $g \geq 2$ with arbitrary $w$. Let $X = \Delta_{g,w}$ and $P,Q \subset X_0$ be properties with $P = \emptyset$ and $Q = \{\B\}$. Then there is a deformation retract from $(\Delta_{g,w})_Q \searrow (\Delta_{g,w})_{Q,0}$.
\label{lemma: prop 4.11}
\end{lem}

\begin{proof}
As is discussed in the proof of Theorem 1.1(2) in \cite{CGP2}, this lemma reduces to the fact that every graph has a canonical maximal uncontraction by $1$-ends. Indeed, the argument given in \cite{CGP2} applies in this case as well.
\end{proof}

We consider the loop-weight locus of $\Delta_{g, w}$ parametrizing tropical curves with loops or vertices of positive weight, denoted by $\Delta_{g, w}^{\mathrm{lw}}$, and show that it is contractible.
\begin{thm}
The subcomplex $\Delta_{g,w}^\mathrm{lw}$ of $\Delta_{g,w}$ is contractible.
\end{thm}

\begin{proof}
We first note that $\Delta_{g,w}^\mathrm{lw}$ is indeed a subcomplex of $\Delta_{g,w}$, as the property of having a loop or vertex of positive weight is closed under edge contraction. When $g=1$ and $\sum_{i=1}^n w_i \leq 1$, the subcomplex $\Delta_{g,w}^\mathrm{lw}$ contains a single point, so the statement is trivially true. Let us assume for the rest of the proof that either $g \geq 2$ or $\sum_{i=1}^n w_i > 1$ when $g=1$.

By Lemma \ref{lemma: prop 4.11}, there is a deformation retract from $(\Delta_{g,w})_Q \searrow (\Delta_{g,w})_{Q,0}$. Observe that $(\Delta_{g,w})_Q$ is exactly the loop-weight locus $\Delta_{g,w}^\mathrm{lw}$. Indeed, if $\G \in (\Delta_{g,w})_Q$, then $G$ is a contraction of a tropical curve with an 1-end, so $\G$ must have a loop or vertex with positive weight; if $\G \in \Delta_{g,w}^\mathrm{lw}$, then $\G$ has a nontrivial uncontraction by 1-ends, so $\G$ is in the closure of a $Q$-simplex. On the other hand, the subcomplex $(\Delta_{g,w})_{Q,0}$ is the closure of the locus of tropical curves that have the following combinatorial type: the graph with a central vertex $v$ and $g$ bridges to vertices of weight 1, with all markings concentrated on $v$. So the geometric realization of $(\Delta_{g,w})_{Q,0}$ is the quotient of a $(g-1)$-simplex by the action of its automorphism group, which is contractible. Therefore, the locus $\Delta_{g,w}^\mathrm{lw}$ is contractible.
\end{proof}

\noindent Next, we show the slightly larger subcomplex $\Delta_{g,w}^\mathrm{mlw}$ that contains $\Delta_{g,w}^\mathrm{lw}$ is again contractible.

\begin{thm}
The subcomplex $\Delta_{g,w}^\mathrm{mlw}$ of $\Delta_{g,w}$ parametrizing tropical curves with loops, vertices of positive weight, or multiple edges is contractible.
\end{thm}

\begin{proof}
Note that contracting an edge in a graph with multiple edges results in a graph with loops or multiple edges, so $\Delta_{g,w}^\mathrm{mlw}$ is a subcomplex. Our proof is parallel to the proof of Theorem 6.1 of \cite{allcock}, which relies on the fact that if $Y$ is contractible and $f: Y \to Z$ is continuous, then the inclusion of $Z$ into the mapping cone of $f$ is a homotopy equivalence. The subcomplex $\Delta_{g,w}^\mathrm{mlw}$ is obtained from the subcomplex $\Delta_{g,w}^\mathrm{lw}$ as an iterated mapping cone through a series of maps from quotients of spheres $S^{p-1}/\Aut(\mathbf{G})$, where $\mathbf{G}$ ranges over all graphs with multiple edges. Such quotients are contractible by \cite[Proposition 5.1]{allcock}, so the inclusion of $\Delta_{g,w}^\mathrm{lw}$ into $\Delta_{g,w}^\mathrm{mlw}$ is a homotopy equivalence.
\end{proof}

\noindent Finally, we prove the main theorem of the section.
\begin{thm}
Let $g, n \geq 1$ and $w \in (\Q \cap (0, 1])^n$. Then $\Delta_{g, w}$ is simply connected.
\end{thm}

\begin{proof}
Observe the 1-skeleton of $\Delta_{g,w}$ is contained in the contractible subcomplex $\Delta_{g,w}^\mathrm{mlw}$. By Theorem 3.1 of \cite{allcock}, there is a surjection from $\pi_1(\Delta_{g,w}^{(1)},x) \to \pi_1(\Delta_{g,w},x)$; that is, $\pi_1(\Delta_{g,w},x)$ is generated by loops in the $1$-skeleton. Since $\Delta_{g,w}^{(1)}$ is contained in a contractible subcomplex, all loops in $\Delta_{g,w}^{(1)}$ are homotopically trivial in $\Delta_{g, w}$, and we conclude that $\pi_1(\Delta_{g,w},x)$ is trivial.
\end{proof}

\section{The Euler characteristic of $\Delta_{g, w}$}
\label{sec:euler}
    Let $M_{g, w}$ be the coarse moduli space of $\mathcal{M}_{g, w}$. In this section we exhibit a useful decomposition of the class $[M_{g,w}]$ in terms of classes $[M_{g, r}]$ in the Grothendieck group of varieties. Using the fact that the virtual Poincar\'{e} polynomial is an Euler-Poincar\'{e} characteristic, this allows us to deduce the formula of Theorem \ref{thm:main-euler}.
    
    \subsection{The Grothendieck group of varieties and Euler-Poincar\'{e} characteristics}
    Let $k$ be a field. We denote by $K_0(\Var/k)$ the \textit{Grothendieck group of varieties}. This group is the quotient of the free abelian group on $k$-varieties by relations of the form
    \[
    [X] = [X \smallsetminus Y] + [Y],
    \] when $Y$ is a closed subvariety of $X$. Such relation are called the \textit{cut-and-paste} relations. The additive identity is $[\varnothing]$. An \textit{Euler-Poincar\'{e} characteristic} of $K_{0}(\Var/k)$ is a group homomorphism 
    \[
    \chi: K_{0}(\Var/k) \to A
    \] to an abelian group $A$. That is, for any closed subvariety $Y$ of $X$, 
    \[
    \chi([X]) = \chi([Y]) + \chi([X \smallsetminus Y]).  
    \]
    See \cite{Craw2004AnIT, loeser2009seattle}. 
    Specializing to $k = \C$, one example of an Euler-Poincar\'{e} characteristic is given by the \textit{virtual Poincar\'{e} polynomial}, which is the group homomorphism ${K_{0}(\Var/\C) \to \Z[t]}$ defined by the formula
     \[P_X(t) = \sum_{m = 0}^{2d} (-1)^m \chi^m_c (X) t^m, \]
    where $d = \dim X$ and
    \[\chi^m_c(X) := \sum_{j = 0}^{2d}(-1)^j \dim \gr^W_{m} H^j_c(X; \Q).\] 
    \subsection{The stratification of $M_{g,w}$} Let $g \ge 0$, $n \ge 1$ and $w \in (\Q \cap (0, 1])^n$ such that 
    \[
    2g - 2 + \sum_{i=1}^n w_i > 0. 
    \]
    
    \noindent To describe a stratification of $M_{g,w}$, we say that a set partition of $[n]$
    \[ \calP = P_1 \sqcup \cdots \sqcup P_r \vdash [n] \]
    is \textbf{\textit{$w$-admissible}} if $ \sum_{i \in P_j} w_i \leq 1  $
    for all $1 \leq j \leq r$. Given such a partition, we write $\mathcal{P} \vdash_w [n]$. We set $N_{r, w}$ to be the number of $w$-admissible partitions of $[n]$ with $r$ parts.
    
    \begin{prop}
    \label{prop:Mgw-cut-and-paste}In the Grothendieck group of $k$-varieties $K_{0}(\Var/k)$,
        \[
        [M_{g,w}] = \sum_{r = 1}^{n} N_{r, w} [M_{g, r}].
        \]
    \end{prop}
    
    \begin{proof}
        The locus $M_{g,w}$ parameterizes irreducible smooth curves of genus $g$ with $n$ markings, such that whenever $\sum_{i \in S} w_i \le 1$ for some $S \subseteq [n]$, the markings indexed by $S$ are allowed to coincide. 
        Given a $w$-admissible partition \[ \calP = P_1 \sqcup \cdots \sqcup P_r \vdash_{w} [n], \] 
        we define
        \[Z_{\calP} := \{(C, p_1, \ldots, p_n) \in M_{g, w} \mid p_i = p_j \text{ if and only if }i, j \in P_s \text{ for some } s \in [r]\}.   \]
        Then $Z_{\calP} \cong M_{g, r}$. As $\calP$ ranges over all $w$-admissible partitions of $[n]$, the loci $Z_{\calP}$ form a locally closed stratification of $M_{g, w}$. Hence in the Grothendieck group, by \cite[Proposition 1.1]{mustata}, we have
        \begin{align*}
            [M_{g, w}] = \sum_{\mathcal{P} \vdash_w [n]} [Z_{\mathcal{P}}] = \sum_{r = 1}^{n} \sum_{\substack{ \mathcal{P} \vdash_w [n] \\|\mathcal{P}| = r}} [Z_\mathcal{P}] = \sum_{r = 1}^{n} N_{r, w}[M_{g, r}],
        \end{align*}
        as claimed.
    \end{proof}
    
    \begin{rem}
        In \cite[Section 4]{Bini2005EulerCO}, Bini and Harer computed the Euler characteristics of $\calM_{g, n}$ for $2g -2 + n  > 0$. Our decomposition in Proposition \ref{prop:Mgw-cut-and-paste} can be used to give the Euler characteristic of $\calM_{g, w}$ in terms of those of $\calM_{g, r}$ for $0 < r \leq n$. 
    \end{rem}
    
    We now record two corollaries which amount to the calculation of the numbers $N_{r, w}$ for special values of $w$.
    \begin{cor} 
        Let $w$ be \textit{heavy/light}, i.e. $w = (1^{(n)}, \varepsilon^{(m)})$ for $m \ge 2$ and $0 < \varepsilon < 1/m$ satisfying $2g - 2 + n \ge 0.$
        Then 
        \[
        [M_{g,w}] = \sum_{r = 1}^{m} S(m, r) [M_{g, n  + r}],
        \] where $S(m, r)$ is the \textit{Stirling number of the second kind}, or the number of $r$-partitions of a $m$-set.  
    \end{cor}
    
    Furthermore, the $m$-restricted Stirling number of the second kind for $n, r \ge 1$ is defined to be the number of partitions of an $n$-set into $r$ nonempty subsets, each of which has at most $m$ elements, and is denoted by \[\begin{Bmatrix}
        n\\
        r
        \end{Bmatrix}_{\leq m}.\] For the generating function and other recurrence relations of the $m$-restricted Stirling numbers; see \cite{KOMATSU16}. Then we have the following. 
    \begin{cor}
    \label{cor:m-res Stirling}
        Let $w=((1/m)^{(n)})$ for $n > m > 1$ such that $2g - 2 + n/m > 0.$ Then \[[M_{g,w}] = \sum_{r=\left\lceil \frac{n}{m}\right\rceil}^{n} \begin{Bmatrix}
        n\\
        r
        \end{Bmatrix}_{\leq m}[M_{g,r}].\]
    \end{cor}
    \begin{cor}
        Let $w = ((1/2)^{(n)})$ for $n \ge 1$ such that 
        $2g - 2 + n/2 > 0$. 
        Then 
        \[[M_{g,w}] = \sum_{r=\left\lceil \frac{n}{2}\right\rceil}^{n} \frac{\prod\limits_{i=0}^{n-r-1}\binom{n-2i}{2}}{(n-r)!} [M_{g,r}].
        \]
    \end{cor}
        
    \begin{proof}
    A $((1/2)^{(n)})$-admissible partition having $r$ parts must consist of exactly $(n-r)$ subsets of size two and singletons otherwise. Therefore, we have \[\begin{Bmatrix}
    n\\
    r
    \end{Bmatrix}_{\leq 2} = \frac{\prod\limits_{i=0}^{n-r-1}\binom{n-2i}{2}}{(n-r)!}.\] Plug this into Corollary \ref{cor:m-res Stirling}, we obtain the desired expression.
    \end{proof}

    \subsection{The Euler characteristics of $\Delta_{g, w}$ and $\calM_{g,w}$}
    We can now exploit the additivity of Euler-Poincar\'{e} characteristics and the connection between $\calM_{g,w}$ and $\Delta_{g, w}$ to prove Theorem \ref{thm:main-euler}. For a complex algebraic variety (or stack) $X$ of dimension $d$, let $\chi^\mathrm{tw}$ be the top weight Euler characteristic, defined as 
    \[
    \chi^\mathrm{tw}(X) := \sum_{i= 0}^{2d} (-1)^{i} \dim \gr^{W}_{2d}H^{i}(X; \Q), 
    \]
    and for any space $Y$, let $\widetilde{\chi}(Y)$ be the reduced Euler characteristic. Recall from ~\cite[Theorem 5.8]{CGP1} that the isomorphism
    \[ \gr^{W}_{6g - 6 + 2n}H^{6g - 6 + 2n - k}(\mathcal{M}_{g, n} ; \Q) \cong \widetilde{H}_{k - 1}(\Delta_{g, n} ; \Q) \]
    is a special case of the isomorphism
    \[\gr^W_{2d}H^{2d - k}(\mathcal{X} ; \Q) \cong \widetilde{H}_{k - 1} (\Delta(\mathcal{X} \subset \overline{\mathcal{X}} ); \Q) \]
    whenever $\mathcal{X}$ is a smooth and separated $d$-dimensional DM stack over $\C$, $\overline{\mathcal{X}}$ is a smooth normal crossings of $\mathcal{X}$, and $\Delta(\mathcal{X} \subset \overline{\mathcal{X}})$ is the dual complex of the normal crossings divisor $\overline{\mathcal{X}} \smallsetminus \mathcal{X}$.
    \begin{lem}
    \label{lem:top-weight-and-reduced-euler}
    Let $\mathcal{X}$ be a smooth, separated DM stack over $\C$ and let $\overline{\mathcal{X}}$ be a smooth normal crossings compactification of $\mathcal{X}$. Then 
        \[
        \chi^\mathrm{tw}(\mathcal{X}) = - \widetilde{\chi}(\Delta(\mathcal{X} \subset \overline{\mathcal{X}})).
        \]
    \end{lem}
    
    \begin{proof}
       Let $d = \dim \mathcal{X}$. We have
        \begin{align*}
            \chi^\mathrm{tw}(\mathcal{X}) &= \sum_{i= 0}^{2d} (-1)^{i} \dim \gr^{W}_{2d}H^{i}(\mathcal{X}; \Q) \\
            &= \sum_{i = 0}^{2d} (-1)^{i} \dim \widetilde{H}_{2d-i-1}(\Delta(\mathcal{X} \subset \overline{\mathcal{X}}); \Q) \\
            &= -\widetilde{\chi}(\Delta(\mathcal{X} \subset \overline{\mathcal{X}})). 
        \end{align*}
    \end{proof}
    
   We now observe that the weight $0$ compactly supported Euler characteristic is a motivic invariant; indeed the following lemma is proven upon realizing that for a complex algebraic variety $X$, we have that $\chi^{0}_c(X) = P_X(0)$, where $P_X$ is the virtual Poincar\'{e} polynomial.
    
    \begin{lem}\label{lem:TopWeightAdditive}
      The weight $0$ compactly supported Euler characteristic
      \[\chi^{0}_c : K_0(\mathrm{Var}/\C) \to \Z  \]
      is an Euler-Poincar\'{e} characteristic.
    \end{lem}
    We can now prove Theorem \ref{thm:main-euler}, restated here.
    \begin{thm1.2}
    \textit{Let $W = W_1 \subset \cdots \subset W_{6g-6+2r} \subseteq H^{\ast}(\calM_{g, r}; \Q)$ be the weight filtration of the rational singular cohomology of the moduli stack $\calM_{g, r}$ and denote by $\chi^{W}_{6g-6+2r}$ the Euler characteristic of the top graded piece
    \[\gr^W_{6g - g + 2r}H^*(\mathcal{M}_{g, r}; \Q) = W_{6g - 6 + 2r}/W_{6g - 7 + 2r}  \]
    of the weight filtration. Then
    \[\chi(\Delta_{g, w}) = 1 -  \sum_{r = 1}^{n} N_{r, w} \cdot \chi^{W}_{6g - 6 + 2r} (\mathcal{M}_{g, r}).  \]}
    \end{thm1.2}
    
    \begin{proof}[Proof of Theorem \ref{thm:main-euler}]
    By Proposition \ref{prop:Mgw-cut-and-paste} and Lemma \ref{lem:TopWeightAdditive}, we have
    \[\chi^0_c(M_{g,w}) = \sum_{r = 1}^{n} N_{r, w} \chi^0_c(M_{g,r}). \]
    By \cite[Proposition 36]{behrend2004cohomology} and \cite[Theorem 4.40]{edidin2010equivariant}, 
    the coarse moduli scheme $\mathcal{X} \to X$ of a DM stack $\mathcal{X}$ induces an isomorphism of rational cohomology $H^{\ast}(\calX; \Q) \cong H^{\ast}(X; \Q)$, which is an isomorphism of mixed Hodge structures.
    Therefore, $\chi_c^0(\calX) = \chi_c^0(X)$ and $\chi^{\mathrm{tw}}(\calX) = \chi^{\mathrm{tw}}(X)$.
    For a smooth Deligne-Mumford stack $\calX$ of dimension $d$, the Poincar\'{e} duality pairing \[H^j_c(\calX; \Q) \times H^{2d - j}(\calX; \Q) \to \Q\] induces a perfect pairing of graded pieces
    \[ \gr_{m}^W H^{j}_c(\calX; \Q) \times \gr_{2d - m}^{W}H^{2d - j}(\calX; \Q) \to \Q  \]
    for $0 \leq m \leq 2j$; see ~\cite[Theorem 6.23]{peters2008mixed}. Thus we can write
    \begin{align*}
        \chi^0_c(\calX) &= \sum_{j = 0}^{2d}(-1)^j \dim \gr^W_{2d} H^{2d - j}(\calX; \Q) \\&= \sum_{j = 0}^{2d}(-1)^{2d - j} \dim \gr^W_{2d} H^{2d - j}(\calX; \Q).
    \end{align*}
    In particular, it follows that
    \[ \chi^0_c(\calX) = \chi^\mathrm{tw}(\calX). \]
    Since $\calM_{g, r}$ and $\calM_{g, w}$ are smooth Deligne-Mumford stacks and $\Delta_{g, w} = \Delta(\calM_{g, w} \subset \overline{\calM}_{g, w})$, the result now follows from Lemma \ref{lem:top-weight-and-reduced-euler} and the fact that $\widetilde{\chi}(\Delta_{g, w}) = \chi(\Delta_{g, w}) - 1$.
    \end{proof}

\bibliographystyle{alpha}
\bibliography{bibliography}
\end{document}